\documentclass[11pt]{amsart}
\usepackage{cases,amsmath,amssymb, hyperref}

\theoremstyle{plain}
\newtheorem{theorem}                {Theorem}      [section]
\newtheorem*{theorem*}                {Theorem \ref{thm:appl}}
\newtheorem{proposition}  [theorem]  {Proposition}
\newtheorem{corollary}    [theorem]  {Corollary}
\newtheorem{lemma}        [theorem]  {Lemma}

\theoremstyle{definition}

\newtheorem{remark}       [theorem]  {Remark}

\DeclareMathOperator{\trace}{trace} 

\DeclareMathOperator{\Div}{div} 
 
\DeclareMathOperator{\Ricci}{Ricci}
\DeclareMathOperator{\Scal}{Scal}
\DeclareMathOperator{\Vol}{Vol}

\DeclareMathOperator{\grad}{grad}
\DeclareMathOperator{\Int}{int}
\numberwithin{equation}{section}

\begin{document}

\title[ Unique Continuation Property for Biharmonic Hypersurfaces in Spheres ]{Unique Continuation Property for Biharmonic Hypersurfaces in Spheres }

\author{Hiba Bibi}
\author{Eric Loubeau}
\author{Cezar Oniciuc}

\thanks{This work was supported by a grant of the Romanian Ministry of Research and Innovation,
CCCDI-UEFISCDI, project number PN-III-P3-3.1-PM-RO-FR-2019-0234 / 1BM / 2019, within
PNCDI III and the PHC Brancusi 2019 project no 43460 TL}

\address{Univ. Brest, CNRS UMR 6205, LMBA, F-29238 Brest, France}\email{Hiba.Bibi@etudiant.univ-brest.fr}\email{Eric.Loubeau@univ-brest.fr}

\address{Faculty of Mathematics\\ Al. I. Cuza University of Iasi\\
Bd. Carol I, no. 11 \\ 700506 Iasi, Romania} \email{oniciucc@uaic.ro}

\subjclass[2010]{58E20, 53C43, 31B30}

\begin{abstract}
We study  properties of non-minimal biharmonic hypersurfaces of spheres. The main result is a CMC Unique Continuation Theorem for biharmonic hypersurfaces of spheres. We  then deduce new rigidity theorems    to support the  Conjecture that   biharmonic submanifolds of  Euclidean spheres must be of constant mean curvature.
\end{abstract}

\keywords{Biharmonic Submanifolds, Constant Mean Curvature, Unique Continuation}

\maketitle

%%%%%%%%
\section{Introduction} 
The study of biharmonic maps was introduced by G.-Y. Jiang  in the mid-80's and, in \cite{Jiangmaps},  he defined biharmonic maps as critical points of the bienergy functional 
$$E_{2}:C^{\infty}(M, N)\to \mathbb{R}, \ \ \ \ E_{2}(\varphi)=\int_{M}\vert \tau(\varphi)\vert ^{2} \ dv_{g},$$
as suggested first by J. Eells and L. Lemaire in  \cite{Eells-Lemaire}. Hence, biharmonic maps come from a variational problem,  generalizing the well-known harmonic maps. Using a simple Bochner formula,
 G.-Y. Jiang proved that  biharmonic maps from a compact manifold to a non-positively curved  space is harmonic,  so the first interesting     target manifold is the Euclidean sphere. By definition,  biharmonic submanifolds are isometric immersions which are biharmonic maps.
 
Independently, in \cite{B-Y. Chen},  B.-Y. Chen  defined  biharmonic submanifolds of the Euclidean space as isometric immersions with harmonic mean curvature vector field, or alternatively, the components of the immersion are   biharmonic functions. In \cite{Ishikawa, J2} it was  proved that biharmonic surfaces in $\mathbb{R}^{3}$ are minimal. This has led to Chen's Conjecture \cite{43}: \textit{Biharmonic submanifolds of Euclidean spaces are minimal}. Some  particular subcases have been proved for example in \cite{Akutagawa and Maeta, Dimitric, Hasanis}.

For  hypersurfaces of the Euclidean sphere, it is natural and useful to split the Euler-Lagrange equation into its tangential and normal components  and to rewrite the biharmonic equation as follows:  \begin{eqnarray*}
\left\lbrace
\begin{array}{ccc}
\Delta f&=& (m-\vert A \vert ^{2})f\\\
A(\grad f)&=&-\frac{m}{2}f \grad f  ,
\end{array}\right.
\end{eqnarray*}
see \cite{B-Y. Chen, Chen,  Oniciuc 2002} (refer to Section \ref{K} for notations).

The first example of a  non-minimal biharmonic hypersurface (due to  G.-Y. Jiang in \cite{Jiangmaps}) is the Generalized Clifford torus $\mathbb{S}^{m_{1}}(\frac{1}{\sqrt{2}})\times \mathbb{S}^{m_{2}}(\frac{1}{\sqrt{2}})$ in $\mathbb{S}^{m_{1}+m_{2}+1}$, with $m_{1}\neq m_{2}$. Recall that, when $m_{1}=m_{2}$, this  Clifford torus is minimal.
In \cite{Oniciuc 2001}, the authors observed that the 45-th parallel  $\mathbb{S}^{m}(\frac{1}{\sqrt{2}})$ in $\mathbb{S}^{m+1},$ which is  umbilical, is another non-minimal biharmonic hypersurface. These two examples have constant mean curvature (CMC) and motivate the following
Conjecture \cite{BMO}: \textit{Any biharmonic submanifold in an Euclidean sphere is CMC}.

This
Conjecture was settled for spheres of dimension 3 in \cite{Oniciuc 2001} and  dimension 4 in \cite{2010}. 
For  arbitrary dimensions, Y. Fu and M.-C. Hong  \cite{fu-hong} proved the Conjecture when the scalar curvature is constant and the number of principal curvatures is at most 6, while S. Maeta and Y.-L. Ou  \cite{Maeta and Ou}
proved it for compact  hypersurfaces with constant scalar curvature. In \cite{ Luo-Maeta, Maeta1}, the authors use new Liouville-type theorems to prove special cases of the Conjecture.

This paper is a contribution to this  Conjecture and  gives rigidity results based on a new technique of  unique continuation theorem (UCT). In \cite{Unique Continuation}, a similar approach was used to prove that if a  biharmonic map is harmonic on an open subset, then it must be  harmonic everywhere.

Inspired by this work of V. Branding and C. Oniciuc, and relying on the UCT of N. Aronszajn  \cite{Aronszajn}, our objective in this article is not to show minimality, but rather to prove the weaker condition of CMC. Using a gradient inequality between  the norm of $A$ and the mean curvature we show that, for  proper-biharmonic  hypersurfaces (i.e. non-minimal) in a sphere, locally CMC implies globally CMC.

In Section \ref{44}, we exploit this UCT property to prove new rigidity results, and in Theorem ~\ref{M}, use an integral condition involving both the scalar and the mean curvatures to force biharmonic hypersurfaces to be CMC. This extends the main result of S. Maeta and Y.-L. Ou \cite{Maeta and Ou}
to non-constant scalar curvature, while relying on a different technique of proof.

\subsection{Conventions}

Manifolds will be assumed to be connected, oriented and without boundary, but  unless stated explicitly, they are not assumed to be compact or complete.

Throughout this paper all manifolds, metrics and maps are taken to be of class $C^{\infty}$, and we adopt the Einstein summation convention.

Let $\varphi:M^{m} \to N^{n}$ between two Riemannian manifolds, the rough Laplacian acting on  sections of the pullback bundle $\varphi^{-1}(TN)$ is given by 
$$ \Delta^{\varphi}= - \trace (\nabla^{\varphi})^{2}=-\trace(\nabla^{\varphi} \nabla^{\varphi}-\nabla^{\varphi}_{\nabla}),$$
where $\nabla^{\varphi}$ is the pullback  connection, for functions  
 $\Delta  = - \trace \nabla \grad . $

Our convention for the curvature tensor field is 
 $$R(X, Y)=[\nabla_{X}, \nabla_{Y}]-\nabla_{[X, Y]}.$$ 
\section{Preliminaries}\label{K}
Let $\varphi:M^{m}\hookrightarrow \mathbb{S}^{m+1}$ be a hypersurface, which we assume, without loss of generality, to be oriented.
Let $\eta \in C(NM)$ be a globally defined unit normal vector field 
and
$A$ the  shape operator
$$A_{\eta}(X)=-\nabla _{X}^{\mathbb{S}^{m+1}} \eta,$$
where $X\in C(TM)$ and  $\nabla^{\mathbb{S}^{m+1}}$ is the Levi-Civita connection on $\mathbb{S}^{m+1}.$
Then, the mean  curvature is 
$$f=\frac{1}{m} \trace A.$$
In general, $f$ can take positive and negative values.

Let $H=f \eta$ be the mean curvature vector field, 
so that $M$ is  minimal when $H=0.$

The second fundamental form $B\in C(\odot^{2}T^{*}M\otimes NM)$ is
$$\langle B(X,Y),\eta\rangle=\langle A(X),Y\rangle,$$
and
\begin{eqnarray*}
\tau(\varphi)=\trace B
= mf \, \eta .
\end{eqnarray*}

It is known that \cite{ Unique Continuation,  MONI}, for a proper-biharmonic map $\varphi:M\to N$, the subset $\{ p\in M : \tau(\varphi)(p)\neq 0\}$ is open and dense in $M$.
Thus, 
$$\Omega=\{ p \in M : f(p)\neq 0\}  $$ 
is open and dense in $M$. 
Note that this subset can have several connected components.
\begin{lemma}(J.-H.~Chen \cite{Chen})
\label{lemma I}
Let $ \varphi :M^{m}\hookrightarrow \mathbb{S}^{m+1}$ be a proper-biharmonic  hypersurface.  Then, at points where $\grad f \neq 0$, we have
$$ \vert A \vert ^{2} \geq \frac{m^{2}(m+8)}{4(m-1)}f^{2}.$$
\end{lemma}
In low dimensions Lemma {\ref{lemma I}} has the following direct consequence.
\begin{proposition}\label{P}
Let $ \varphi :M^{m}\hookrightarrow \mathbb{S}^{m+1}$ be a proper-biharmonic  hypersurface. Assume that $m=3$ or $m=4$, $\Scal ^{M}>m(m-1)$,
then M has constant mean curvature.
\end{proposition}
\begin{proof}
Assume that $M$ does not have constant mean curvature, then there exists $p_{0} \in M$ such that $(\grad f) (p_{0})\neq 0$.
By Lemma \ref{lemma I},  we have $$\vert A (p_{0})\vert ^{2}\geq \frac{m^{2}(m+8)}{4(m-1)}f^{2}(p_{0}).$$
On the other hand, taking traces in the Gauss Equation (see for example \cite{Docarmo}), we have 
\begin{eqnarray*}
\Scal^{\mathbb{S}^{m+1}} &=& \Scal^{M}+\vert A \vert^{2}-m^{2}f^{2}+ 2 \Ricci^{\mathbb{S}^{m+1}}(\eta ,\eta),
\end{eqnarray*}
where $\eta$ is the unit normal vector field. Thus we have
$$\vert A \vert ^{2}=m(m-1)+m^{2}f^{2}-\Scal^{M},$$ as $\Scal^{M}>m(m-1)$ we obtain 
$$\frac{m^{2}(m+8)}{4(m-1)}f^{2}(p_{0}) \leq \vert A (p_{0})\vert ^{2}< m^{2}f^{2}(p_{0})$$
which forces $m=3$ or $m=4$.

\end{proof}
The conjecture says that any proper-biharmonic hypersurface in $\mathbb{S}^{m+1}$ has constant mean curvature. When $M$ is compact, this conjecture was proved in several cases, under additional hypotheses. When $M$ is not compact, and the additional hypotheses are still satisfied, we only can say that the points where $\grad f \neq 0$, if they exist, cannot form a set with a simple structure.
We present here only one result of this type.
\begin{proposition}
Let $ \varphi :M^{m}\hookrightarrow \mathbb{S}^{m+1}$ be a proper-biharmonic  hypersurface. Assume that $M$ does not have constant mean curvature. If $\vert A \vert^{2} \geq m$, or $\vert A \vert ^{2} \leq m$, then  $W=\{ p \in M:(\grad f )(p)\neq 0\}$ cannot have a connected component $W_{0}$ with the following properties:
\begin{enumerate}
\item $\overline{W_{0}}^{M}$ is compact;
\item the boundary of $W_{0}$ in $M$ is a regular (not necessarily connected) hypersurface of $M$;
\item there exists an open subset $U$ of $M$ such that $\overline{W_{0}}^{M}\subset U$ and $\grad f =0 $ on $U \backslash W_{0}$.
\end{enumerate}
\end{proposition}
\begin{proof}
Assume that $W$ has a connected component $W_{0}$ with the above properties and we argue by contradiction.

Since $\partial W_{0}$ is a regular hypersurface of $M$, we have 
$$ \Int (U\backslash W_{0})=U\backslash \overline{W_{0}}^{U}=U\backslash \overline{W_{0}}^{M} =U\backslash (W_{0} \cup \partial W_{0}) \neq \emptyset,$$
otherwise $U=W_{0} \cup \partial W_{0}$ is closed in $M$, so $U=M$, i.e. $M=W_{0}\cup \partial W_{0}$ is a manifold with  boundary; and 
$$ \overline{\Int (U\backslash W_{0})}^{U}=U\backslash W_{0}.$$
On $\Int (U \backslash W_{0})$, that may have several connected components, $\grad f =0$, $f\neq0$ and $\vert A \vert ^{2}=m.$

Assume now that $\vert A \vert ^{2}\leq m.$
It was proved in \cite{Chen}, see also \cite[Inequality (1.74)]{Thesis} that on $M$ we have
$$
\frac{1}{2} \Delta (\vert \grad f \vert ^{2}+ \frac{m^{2}}{8}f^{4}+f^{2})+\frac{1}{2} \Div(\vert A \vert ^{2}\grad f ^{2}) \leq \frac{8(m-1)}{m(m+8)}(\vert A \vert ^{2}-m)\vert A \vert ^{2}f^{2}.
$$
Equivalently, 
\begin{align}\label{Z}
-\Div Z\leq\frac{8(m-1)}{m(m+8)}(\vert A \vert ^{2}-m)\vert A \vert ^{2}f^{2}\leq 0,
\end{align}
where
\begin{align*}
Z=\frac{1}{2}\grad (\vert \grad f \vert ^{2}+\frac{m^{2}}{8}f^{4}+f^{2})-\frac{1}{2}\vert A \vert ^{2}\grad f ^{2}.
\end{align*}
Since $Z=0$ on $\Int (U\backslash W_{0}),$ it follows that $Z=0$ on $U\backslash W_{0}$ and so on $\partial W_{0}.$
Integrating  Inequality (\ref{Z}) on $\overline{W_{0}}^{M}$ and using the Divergence Theorem, as $Z=0$ on $\partial W_{0}$, we obtain $(\vert A \vert ^{2}-m)\vert A \vert ^{2}f^{2}=0$ on $\overline{W_{0}}^{M}.$ As in \cite{Thesis}, we obtain $\vert A \vert ^{2}=m $ on $W_{0}$, and so on $\overline{W_{0}}^{M}$. It follows that $\Delta f=0$ on $\overline{W_{0}}^{M}$.

Furthermore, we integrate $\Delta f^{2}$ on $\overline{W_{0}}^{M}$, and since $\grad f ^{2}=0$ on $\partial W_{0}$, we obtain $\grad f =0$ on $\overline{W_{0}}^{M}$ which is impossible.

The case $\vert A \vert ^{2}\geq m $ is easy to prove as 
\begin{align*}
\frac{1}{2}\Delta f^{2}=(m-\vert A \vert ^{2})f^{2}-\vert \grad f \vert ^{2}\leq 0
\end{align*}
on $M$, and integrating on $\overline{W_{0}}^{M}$ we obtain again $\grad f =0$ on $\overline{W_{0}}^{M}$.

\end{proof}
In the next section (see Corollary \ref{corollary1}) we will see that under  a stronger hypothesis, i.e. $\vert A \vert ^{2}$ is constant, the points of a   non-CMC  proper-biharmonic hypersurface where $\grad f \neq 0$ form an open dense subset of $M$.

Before stating the last result of this section, we need to recall some well-known facts about the smoothness of the principal curvatures.

Let $ \varphi :M^{m}\hookrightarrow \mathbb{S}^{m+1}$ be a hypersurface with  $\lambda_{1}\geq \lambda_{2}\geq \cdots \geq \lambda_{m}$ its principal curvatures, i.e. the eigenvalue functions of the shape operator $A$. The functions $\lambda_{i}$ are continuous on M for all $i=1,\ldots,m$. The set of points where the numbers of  distinct principal curvatures is locally constant is a set $M_{A}$ that is open and dense in $M$. On a non-empty connected component of $M_{A}$, which is open in $M_{A}$, and so in $M$, the number of distinct principal curvatures is constant. Thus, the multiplicities of the principal curvatures are constant, and so, on that connected component, $\lambda_{i}$'s are smooth and $A$ is (smoothly) locally diagonalizable (see \cite{Nomizu,Ryan1, Ryan2}).
\begin{proposition}\label{PR}
Let $ \varphi :M^{m}\hookrightarrow \mathbb{S}^{m+1}$ be a proper-biharmonic hypersurface. Assume that at any point of $M$ the multiplicity of distinct principal curvatures is at least $2$. Then $M$ has constant mean curvature.
\end{proposition}
\begin{proof}
Assume that $M$ is not CMC and denote $$W:=\{ p\in M: (\grad f )(p)\neq 0\}.$$
Clearly, $W$ is a non-empty open subset of $M$.
Since $M_{A}$ is dense, $W\cap M_{A}\neq \emptyset$, and so $W$ intersects a connected component of $M_{A}$. On that intersection, $\lambda _{i}$'s are smooth, $i=1,\ldots ,m,$ and $A$ is smoothly diagonalizable, i.e. $A(E_{i})=\lambda _{i} E_{i}$, $i=1, \ldots , m, $ where $\{ E_{i}\}_{i=1}^{m}$ is an orthonormal frame field.

From the hypothesis, we can assume for simplicity  that $\lambda_{1}=\lambda_{2}=-\frac{m}{2}f$ and $E_{1}=\frac{\grad f}{\vert \grad f \vert}.$
Since $\langle E_{a}, E_{1}\rangle=0$, we have 
\begin{eqnarray}\label{C}
E_{a}f=0, \  \  a=2,\ldots, m.
\end{eqnarray}
Now, we use the connection equations with respect to the frame field $\{ E_{i} \}_{i=1}^{m}$,
$$ \nabla_{E_{i}} E_{j}=\omega _{j}^{k}(E_{i})E_{k},$$
and we rewrite the Codazzi equation
$$ (\nabla_{E_{i}}A)(E_{j})=(\nabla_{E_{j}}A)(E_{i})$$
as
\begin{align}\label{E1}
(E_{i}\lambda_{j})E_{j}+\sum_{k=1}^{m}
(\lambda_{j}-\lambda_{k})\omega_{j}^{k}
(E_{i})E_{k}=
(E_{j}\lambda_{i})E_{i}+\sum_{k=1}^{m}(\lambda_{i}-\lambda_{k})\omega _{i}^{k}(E_{j})E_{k}.
\end{align}
For $i=1$ and $j=2$ we obtain
\begin{align}
(E_{1}\lambda_{2})E_{2}+\sum_{k=1}^{m}
(\lambda_{2}-\lambda_{k})\omega_{2}^{k}
(E_{1})E_{k}&=
(E_{2}\lambda_{1})E_{1}+\sum_{k=1}^{m}(\lambda_{1}-\lambda_{k})\omega _{1}^{k}(E_{2})E_{k}.\nonumber
\\
&= \sum_{k=1}^{m}(\lambda_{1}-\lambda_{k})\omega_{1}^{k}(E_{2})E_{k}.\label{propeq}
\end{align}
Furthermore, we take the scalar product of the above relation with $E_{2}$, and we obtain 
$$ E_{1}\lambda_{2}=E_{1}\lambda_{1}=0,$$
i.e. $E_{1}f=0.$ Thus, from Equation (\ref{C}) we conclude that $\grad f=0$ which is impossible.

\end{proof}

\section{The unique continuation theorem}
Very little is known on the local properties, in particular
analytical ones, of biharmonic submanifolds in Euclidean spheres.

An essential tool in the analysis of PDE's is a unique continuation
property, which we establish in Theorem~ \ref{theorem1} under a global condition on
the gradients of the norm of the shape operator and mean curvature.

The objective here departs from \cite{Unique Continuation} as the conclusion is
that the manifold has constant mean curvature, instead of the stronger
condition of minimality, but the method is similar and is based on
Aronszajn's unique continuation theorem of 1957 \cite{Aronszajn}.

In Corollaries ~\ref{corollary1} and \ref{corollary2}, the main hypothesis of Theorem ~ \ref{M} is replaced
by more geometrical constraints and allows to extend known results from
the compact to the non-compact cases.
\begin{theorem}\label{theorem1}
Let $ \varphi :M^{m}\hookrightarrow \mathbb{S}^{m+1}$ be a proper-biharmonic  hypersurface. Assume that  there exists a non-negative function h on $M$ such that  $ \vert \grad  \vert A\vert ^{2}\vert\leq h \  \vert \grad  f\vert$ on \textit{M}. If $ \grad  f$ vanishes on a non-empty open connected subset of \textit{M}, then $M$ has constant mean curvature.
\end{theorem}
\begin{proof}
Denote by $V$ the non-empty open connected subset of $M$ where $\grad f=0.$
\\
Consider the subset
$$A_{0}:=\{p \in M: ( \grad f)(p)=0\}.
$$
It is clear that $A_{0}$ is closed, $\Int A_{0}\neq \emptyset$ and $\Int A_{0}$ may have several connected components.
Indeed,
$A_{0}=( \grad  f)^{-1}(\{0\})$ where $\{0\}$ is closed in $TM$ as being the zero section, and $ \grad f :M \to TM$ is continuous, thus $A_{0}$ is closed, and since $V$  is a non-empty open 
subset of $ A_{0}$, we obtain $\Int A_{0}$ is also non-empty.

Assume that $ \partial  ( \Int  A_{0})= \emptyset. $
Then
\begin{eqnarray*}
\emptyset  &=& \partial ( \Int  A_{0})
\\
&=& \overline{( \Int A_{0})}^{M} \cap \  \overline{(M  \backslash \Int A_{0})}^{M}\\&=&\overline{( \Int A_{0})}^{M} \cap \  (M\backslash \Int A_{0}).
\end{eqnarray*}
Now, as $\emptyset  =\overline{( \Int A_{0})}^{M} \cap \  (M  \backslash \Int A_{0})$, we obtain $ \overline{( \Int A_{0})}^{M} \subset \Int  A_{0}$ which implies that $\overline{( \Int A_{0})}^{M}= \Int  A_{0} 
$, thus $\Int {A_{0}}$ is closed in \textit{M}. But $\Int  A_{0}$ is  non-empty open and \textit{M} is connected, 
we conclude that
$ \Int  A_{0} = M $, so $ \Int  A_{0} = A_{0} =M$ and $\grad  f =0$ on \textit{M}.

Assume now that $\partial ( \Int A_{0}) \neq \emptyset$, we will obtain a contradiction. Let $p_{0}\in \partial( \Int  A_{0} )$,  $\partial ( \Int  A_{0} )=\overline{( \Int  A_{0} )}^{M}\backslash \Int A_{0}$, necessarily $p_{0} \notin \Int A_{0}$.
Let \textit{U} be an open subset containing $p_{0}$,  then $ U \cap \Int  A_{0} \neq \emptyset$.

On the other hand, we have
$$ p_{0} \in \partial ( \Int A_{0}) \subset \partial A_{0} ,$$
so 
$$ p_{0} \in \partial A_{0} = \partial (M  \backslash  A_{0}).$$
Since $A_{0}$ is closed in \textit{M},  then
 $ M   \backslash  A_{0}$ is non-empty open in \textit{M}, and so $ p_{0} \notin M \backslash A_{0}.$
Of course, $p_{0} \in \overline{(M \backslash  A_{0})}^{M} $ implies that $ U \cap (M  \backslash   A_{0}) \neq \emptyset. $

In conclusion:
\begin{enumerate}
\item $U \cap  \Int  A_{0}$ is a non-empty open subset of $\Int A_{0}$ that does not contain $p_{0}$, so there exists a non-empty open subset on which $\grad  f=0$.
\item $ U   \cap (M  \backslash   A_{0}) $ is  a non-empty open subset that  does not contain $p_{0},$  and  is  included  in $ M  \backslash  A_{0} $, so there exists a non-empty open  subset on which  $\grad  f \neq 0 $ at any point.
\end{enumerate}
Let $(U,x^{i})_{i=1,\dots, m}$ be   a  local  chart  on  $M$ around  $p_{0}\in \partial( \Int A_{0}).$
Consider an open connected  subset $ D$ in $M$ containing $ p_{0}$, such that  $\overline{D}^{M}$  is  compact   and $ \overline{D}^{M}\subset U$.
Note that $D$ also contains a  non-empty open subset where $ \grad f =0$ everywhere, and a non-empty open subset where $ \grad f \neq 0$ at any point.

As usual, we identify $\grad f \in C(TM)$ with $d\varphi(\grad f)\in C(\varphi^{-1} T\mathbb{S}^{m+1})$, or $$d(i\circ \varphi)(\grad f )\in C((i\circ \varphi)^{-1}T\mathbb{R}^{m+2}),$$ where $i: \mathbb{S}^{m+1}\hookrightarrow \mathbb{R}^{m+2}$ is the canonical inclusion.
Let us write
$ \grad  f =u^{\alpha}e_{\alpha}$,  where $u^{\alpha} \in C^{\infty}(M),  \ \forall \alpha=1, \dots, m+2 $, and 
$\{e_{\alpha}\}_{ \alpha =1}^{m+2}$ is the canonical basis in $ \mathbb{R}^{m+2} $.  
For all $\alpha =1,\dots , m+2 $, the function $u^{\alpha}$ vanishes on $V$.

As $\varphi$ is biharmonic, we have
$$ \Delta f = (m-|A|^{2})f, $$
and taking its differential we obtain
\begin{eqnarray}\label{**}
d  \Delta f = (m-|A|^{2})df-fd(|A|^{2}),
\end{eqnarray}
hence, by the musical isomorphism:
$$(d \Delta f)^{\sharp}= [(m-|A|^{2})df-fd(|A|^{2})]^{\sharp},$$
Since
$(df)^{\sharp}= \grad  f$
 and $ d \Delta ^{\mathrm{Hodge}}= \Delta^\mathrm{{Hodge}}d,$
we can rewrite Equation (\ref{**}) as
\\
$$\Big( \Delta ^{\mathrm{Hodge}}( d f)\Big)^{\sharp}=(m-|A|^{2})\grad f -f \grad  |A|^{2}.$$
On the other hand, by the Weitzenbock formula 
$$ \Big( \Delta ^{\mathrm{Hodge}}( d f )\Big)^{\sharp}=- \trace \ \nabla ^{2} \grad f + \Ricci^{M}( \grad  f), $$
thus 
\begin{align}
-\trace \ \nabla ^{2} \grad f= -\Ricci^{M}( \grad  f)+(m-|A|^{2}) \ \grad f-f \ \grad  |A|^{2}.\label{1}
\end{align}
As  $ \Ricci^{M}( \grad   f)= \Ricci^{M}(u^{\alpha}e_{\alpha})$
and
$$ \grad f= u^{\alpha}e_{\alpha}=u^{\alpha}(e_{\alpha}^{\perp}+e_{\alpha}^{T})=u^{\alpha}e_{\alpha}^{T},$$
where $e_{\alpha}^{\perp}$ and $e_{\alpha}^{T}$ are the normal and the tangential components to $M$ of $e_{\alpha}$ in $\mathbb{R}^{m+2}$ respectively,  we obtain
\begin{align}
\Ricci^{M}( \grad f)= \Ricci^{M}(u^{\alpha}e_{\alpha}^{T})=u^{\alpha}\Ricci^{M}(e_{\alpha}^{T}).\label{2}
\end{align}
On $U$, we combine the second fundamental forms of $M$ in  $\mathbb{S}^{m+1}$ and $\mathbb{S}^{m+1}$ in $\mathbb{R}^{m+2}$ to compute
\begin{eqnarray*}
\nabla_{\frac{\partial}{\partial x^{i}}}^{M} \  \grad f &=&\nabla_{\frac{\partial}{\partial x^{i}}}^{\mathbb{S}^{m+1}} \ \grad f - \ B \Big(\frac{\partial}{\partial x^{i}}, \grad f \Big)
\\&=& \nabla_{\frac{\partial}{\partial x^{i}}}^{\mathbb{R}^{m+2}} \ \grad f + \Big\langle \frac{\partial}{\partial x^{i}},u^{\alpha} e_{\alpha}^{T}\Big\rangle r - B\Big(\frac{\partial}{\partial x^{i}}, \grad f \Big)
\\&=&\nabla_{\frac{\partial}{\partial x^{i}}}^{\mathbb{R}^{m+2}}(u^{\alpha}e_{\alpha})+\Big\langle \frac{\partial}{\partial x^{i}}, u^{\alpha}e_{\alpha}^{T}\Big \rangle r - B\Big(\frac{\partial}{\partial x^{i}}, u^{\alpha}e_{\alpha}^{T}\Big)\\
&=& \frac{\partial u^{\alpha}}{\partial x^{i}} \ e_{\alpha} +u^{\alpha }\nabla_{\frac{\partial}{\partial x^{i}}}^{\mathbb{R}^{m+2}} e_{\alpha}+u^{\alpha}\Big\langle  \frac{\partial}{\partial x^{i}},e_{\alpha}^{T}\Big \rangle r - u^{\alpha}B\Big(\frac{\partial}{\partial x^{i}}, e_{\alpha}^{T}\Big)\\
&=&\frac{\partial u^{\alpha}}{\partial x^{i} }\ e_{\alpha} +u^{\alpha} \Big \langle  \frac{\partial}{\partial x^{i}},e_{\alpha}^{T}\Big \rangle r - u^{\alpha}B\Big(\frac{\partial}{\partial x^{i}}, e_{\alpha}^{T}\Big)\\&=& \frac{\partial u^{\alpha}}{\partial x^{i}} \ e_{\alpha}^{T},
\end{eqnarray*} 
\\
where $r$ is the position vector field on $\mathbb{R}^{m+2}$.
Put
$$Y_{i}=\frac{\partial u^{\alpha}}{\partial x^{i}} \ e_{\alpha}^{T}.$$
For our purposes, it is  convenient to write $Y_{i}$ as $$ Y_{i}=\frac{\partial u^{\alpha}}{\partial x^{i}}\ e_{\alpha}+u^{\alpha} Z_{\alpha,i},$$
where 
$$ Z_{\alpha,i}= \Big\langle \frac{\partial}{\partial x^{i}},e_{\alpha}^{T} \Big\rangle r- B\Big(\frac{\partial}{\partial x^{i}},e_{\alpha }^{T}\Big)$$
is a vector field normal to $M$ in $ \mathbb{R}^{m+2}$.

We repeat this process to obtain, on $U$, the second derivatives of $\grad f $,
\begin{align}
\nabla _{\frac{\partial}{\partial x^{i}}}^{M} \nabla _{\frac{\partial}{\partial x^{j}}}^{M} \grad f &= \nabla _\frac{\partial}{\partial x^{i}}^{M  }Y_{j}{\nonumber}
\\&= \nabla _{\frac {\partial}{\partial x^{i}}} ^ {\mathbb{R}^{m+2}} Y_{j} +\Big \langle \frac{\partial}{\partial x^{i}},Y_{j} \Big\rangle r -B\Big(\frac{\partial}{\partial x^{i}},Y_{j}\Big){\nonumber}
\\&= \nabla_{\frac{\partial}{\partial x^{i}}}^{\mathbb{R}^{m+2}} \Big \{ \frac{\partial u^{\alpha}}{\partial x^{j}}
\ e_{\alpha} + u^{\alpha}Z_{\alpha,j}  \Big \}+\Big\langle \frac{\partial}{\partial x^{i}}, Y_{j}\Big\rangle r -B \Big(\frac{\partial}{\partial x^{i}},Y_{j}\Big){\nonumber}
\\
& =\frac{\partial ^{2}u^{\alpha}}{\partial x^{i}\partial x^{j}} \ e_{\alpha}+\frac{\partial u^{\alpha}}{\partial x^{i}}\ Z_{\alpha,j}+ u^{\alpha}\nabla _{\frac{\partial}{\partial x^{i}}}^{\mathbb{R}^{m+2}}  Z_{\alpha,j} +\Big\langle \frac{\partial}{\partial x^{i}},Y_{j}\Big\rangle r \label{3}
\\
 & \quad-B\Big(\frac{\partial}{\partial x^{i}},Y_{j}\Big).\nonumber
\end{align} 
To compute $\nabla_{\nabla^{M}_{\frac{\partial}{\partial x^{i}}}\frac{\partial}{\partial x^{j}}}^{M}  \grad f, $ on $U$, we have
\begin{eqnarray}
\nabla_{\nabla^{M}_{\frac{\partial}{\partial x^{i}}}\frac{\partial}{\partial x^{j}}}^{M}  \grad f &=& \nabla_{\nabla^{M}_{\frac{\partial}{\partial x^{i}}}\frac{\partial}{\partial x^{j}}}^{\mathbb{R}^{m+2} } \grad f +\Big\langle \nabla _\frac{\partial}{\partial x^{i}}^{M}\frac{ \partial}{\partial x^{j}}, \grad f \Big \rangle r -B \Big(\nabla _{\frac{\partial}{\partial x^{i}}}^{M}\frac{\partial}{\partial x^{j}}, \grad f \Big)
{\nonumber} \\&=& \Big[ \Big( \nabla _{\frac{\partial}{\partial x^{i}}}^{M} \frac{\partial}{\partial x^{j}}\Big)u^{\alpha} \Big]e_{\alpha} +u^{\alpha } \Big\langle \nabla _{\frac{\partial}{\partial x^{i}}}^{M}\frac{\partial }{\partial x^{j}}, e_{\alpha}^{T} \Big\rangle r - u^{\alpha}B\Big( \nabla _{\frac{\partial}{\partial x^{i}}}^{M} \frac{\partial}{\partial x^{j}}, e_{\alpha}^{T} \Big)
{\nonumber} \\&=&\Big[ \Big( \nabla _{\frac{\partial}{\partial x^{i}}}^{M} \frac{\partial}{\partial x^{j}}\Big)u^{\alpha} \Big]e_{\alpha}+ u^{\alpha} \  W_{\alpha,ij},\label{4}
\end{eqnarray}
where
$$ W_{\alpha ,ij}= \Big\langle \nabla _{\frac{\partial}{\partial x^{i}}}^{M}\frac{\partial}{\partial x^{j}}, e_{\alpha}^{T}\Big\rangle r -B\Big(  \nabla _{\frac{\partial}{\partial x^{i}}}^{M}\frac{\partial}{\partial x^{j}}, e_{\alpha} ^{T}\Big)$$
is a vector field normal to $M$ in $\mathbb{R}^{m+2}$.

Replacing (\ref{3}) and (\ref{4}) in (\ref{1}), and using (\ref{2}), we obtain
\begin{eqnarray*}
(\Delta u^{\alpha}) \ e_{\alpha}-g^{ij} \ \frac{\partial u^{\alpha}}{\partial x^{i} } \ Z_{\alpha , j}-g^{ij} \Big\langle \frac{\partial}{\partial x^{i}} , Y_{j}\Big\rangle r+g^{ij}B\Big(  \frac{\partial}{\partial x^{i}}, Y_{j}\Big)-g^{ij}u^{\alpha} \nabla _{\frac{\partial}{\partial x^{i}}} ^{\mathbb{R}^{m+2}} Z_{\alpha,j} \nonumber 
\\ +g^{ij}u^{\alpha}\ W_{\alpha,ij}  = -u^{\alpha} \Ricci ^{M}(e_{\alpha}^{T})
 +(m-\vert A \vert ^{2})u^{\alpha}e_{\alpha}^{T} -f \grad \vert A\vert ^{2},
\end{eqnarray*}
so
\begin{eqnarray}
 \quad \quad (\Delta u^{\alpha})\ e_{\alpha} &=& g^{ij} \ \frac{\partial u^{\alpha}}{\partial x^{i} } \ Z_{\alpha , j}+g^{ij} \Big\langle \frac{\partial}{\partial x^{i}} , Y_{j}\Big\rangle r -g^{ij}B\Big(  \frac{\partial}{\partial x^{i}}, Y_{j}\Big)\label{5}
\\
&\ & 
+ \ g^{ij}u^{\alpha} \nabla _{\frac{\partial}{\partial x^{i}}} ^{\mathbb{R}^{m+2}} Z_{\alpha,j} - g^{ij}u^{\alpha}\Big\langle \nabla _{\frac{\partial}{\partial x^{i}}}^{M}\frac{\partial}{\partial x^{j}}, e_{\alpha}^{T}\Big\rangle r \nonumber
\\
& \ &
+g^{ij} u^{\alpha}B\Big(  \nabla _{\frac{\partial}{\partial x^{i}}}^{M}\frac{\partial}{\partial x^{j}}, e_{\alpha} ^{T}\Big)
- u^{\alpha} \Ricci^{M}(e_{\alpha}^{T})\nonumber
+  (m- \vert A \vert ^{2})u^{\alpha}e_{\alpha}^{T} \nonumber
\\ 
&\ & - \ f \grad \vert A \vert ^{2}.\nonumber 
\end{eqnarray}
But 
\begin{eqnarray}
g^{ij} \ \Big\langle \frac{\partial}{\partial x^{i}}, Y_{j} \Big\rangle r =g^{ij} \ \frac{\partial u^{\alpha}}{\partial x^{j}}  \ \Big\langle \frac{\partial}{\partial x^{i}},e_{\alpha}^{T} \Big\rangle r,\label{6}
\end{eqnarray} 
and
\begin{eqnarray}
g^{ij} \ B\Big(  \frac{\partial}{\partial x^{i}},Y_{j}\Big)=g^{ij} \ \frac{\partial u^{\alpha}}{\partial x^{j}} \ B\Big( \frac{\partial}{\partial x^{i}},e_{\alpha}^{T} \Big), \label{7}
\end{eqnarray}
so replacing (\ref{6}) and (\ref{7}) in (\ref{5}) we obtain 
\begin{align}
(\Delta u^{\alpha})\ e_{\alpha}
&= g^{ij} \ \frac{\partial u^{\alpha}}{\partial x^{i} } \ Z_{\alpha , j}+g^{ij} \ \frac{\partial u^{\alpha}}{\partial x^{j}}  \ \Big\langle \frac{\partial}{\partial x^{i}},e_{\alpha}^{T} \Big\rangle r -g^{ij} \ \frac{\partial u^{\alpha}}{\partial x^{j}} \ B\Big( \frac{\partial}{\partial x^{i}},e_{\alpha}^{T} \Big)\label{8}
\\
& \quad +g^{ij}u^{\alpha} \nabla _{\frac{\partial}{\partial x^{i}}} ^{\mathbb{R}^{m+2}} Z_{\alpha,j} -
g^{ij}u^{\alpha}\Big\langle \nabla _{\frac{\partial}{\partial x^{i}}}^{M}\frac{\partial}{\partial x^{j}}, e_{\alpha}^{T}\Big\rangle r+g^{ij} u ^{\alpha}B\Big(  \nabla _{\frac{\partial}{\partial x^{i}}}^{M}\frac{\partial}{\partial x^{j}}, e_{\alpha} ^{T}\Big) \nonumber
\\
& \quad
-u^{\alpha} \Ricci^{M}(e_{\alpha}^{T})\nonumber
 +(m- \vert A \vert ^{2})u^{\alpha}e_{\alpha}^{T} -f \grad \vert A \vert ^{2}.\nonumber
\end{align}
Thus each term on the right-hand side of Equation (\ref{8}), except for the last one, contains either $\frac{\partial u^{\alpha}}{\partial x^{i}}$ or $u^{\alpha}$.

By the triangle inequality
$$ \vert \Delta u^{\alpha_{0}} \vert \leq \vert( \Delta u^{\alpha})e_{\alpha} \vert,$$
and since all functions and vector fields are smooth on $U$, they are bounded on $\overline{D}^{M}$, and so, on $D$. Using the hypothesis and standard inequalities, we obtain
$$ \vert \Delta u^{\alpha_{0}} \vert \leq C\Big( \sum_{\alpha, i} \Big \vert \frac{\partial u^{\alpha}}{\partial x^{i}} \Big\vert +\sum_{\alpha} \vert u^{\alpha} \vert \Big)$$
on $D$.
Since $u^{\alpha}$ is zero on a non-empty open subset of $D$, by  Aronszajn's unique continuation principle we deduce that $u^{\alpha}$ is equal to zero on $D$, and thus $\grad f $ vanishes on $D.$ This is impossible,
hence the assumption  $\partial ( \Int A_{0})\neq \emptyset $ is false. In conclusion, $ \partial (\Int A_{0})= \emptyset $, and so $ \grad f $ vanishes on the whole of $M$.

\end{proof}
Theorem ~\ref{theorem1} can be rephrased as follows:
\begin{corollary}\label{Cor}
Let $ \varphi :M^{m}\hookrightarrow \mathbb{S}^{m+1}$ be a proper-biharmonic  hypersurface. Assume that there exists a non-negative function $h$ on $M$ such that $\vert \grad \vert A \vert ^{2} \vert \leq h \vert \grad f \vert$ on $M$. Then, either $M$ has constant mean curvature, or the set of points where $\grad f \neq 0 $ is an open dense subset of $M$. 
\end{corollary}
\begin{proof}
Assume that $M$ is not CMC. Let
$$ W:= \{ p \in M : (\grad f )(p) \neq 0 \},$$
be a non-empty open subset in $M$. Assume that $\overline{W}^{M} \varsubsetneq M ,$ then $V= M \setminus \overline{W}^{M} $ is a non-empty, open subset of M 
and $\grad f \vert _{V}=0$, therefore $f$ is constant on a connected component $V_{1}$ of $V$.
As $f$ is constant on $V_{1}$  and $\vert \grad \vert A \vert ^{2} \vert \leq h  \ \vert \grad f \vert$ over $M$, by Theorem ~ \ref{theorem1} we deduce that $f$ is constant on $M$, which is a  contradiction, therefore
 $$\overline{W}^{M}=M.$$  
\end{proof}

The hypothesis on the existence of the function $h$ in Theorem~
\ref{theorem1} can be obtained under natural conditions on $|A|^2$ or the scalar
curvature of $M$.

\begin{corollary} \label{corollary1}
Let $ \varphi :M^{m}\hookrightarrow \mathbb{S}^{m+1}$ be a proper-biharmonic  hypersurface  with $\vert A \vert ^{2}$  constant.   Then, either $M$ has constant mean curvature, or the set of points where $\grad f \neq 0 $ is an open dense subset of $M$. 
\end{corollary}
\begin{proof}
As $ \vert A \vert ^{2}$ is constant,  the condition 
$ \vert \grad  \vert A\vert ^{2}\vert\leq h \  \vert \grad  f\vert$ on \textit{M} is automatically satisfied, thus, by  Corollary ~ \ref{Cor} we conclude.

\end{proof}
\begin{corollary}\label{corollary2}
Let $ \varphi :M^{m}\hookrightarrow \mathbb{S}^{m+1}$ be a proper-biharmonic hypersurface  with constant scalar curvature.  Then, either $M$ has constant mean curvature, or the set of points where $\grad f \neq 0 $ is an open dense subset of $M$.
\end{corollary}
\begin{proof}
By  Proposition \ref{P} we have
 $$ \vert A \vert ^{2} = m(m-1) +m^{2}f^{2}- \Scal^{M},
$$
which implies
$$\vert \grad  \vert A \vert^{2} \vert = 2m^{2}\vert f \vert \  \vert \grad  f \vert.
$$
Therefore, the condition
$$\vert \grad  \vert  A\vert ^{2} \vert \leq  h \vert \grad  f \vert$$
 holds on $M$ and we apply Corollary ~\ref{Cor} to conclude.

\end{proof}
\begin{remark}
$  $
\begin{enumerate}
\item Corollaries \ref{corollary1} and \ref{corollary2} are meaningful because $M$ is not assumed to be compact:
\begin{enumerate}
\item
A direct consequence of J.-H.~Chen's result is that if $M$ is compact and $\vert A \vert ^{2}$ is constant, then $\grad f$ vanishes on the whole manifold $M$(see \cite{Thesis}).
\item If M is compact and its scalar curvature is constant,  Maeta and Ou show in \cite{Maeta and Ou} that $f$ is constant.
\end{enumerate}
Therefore, Corollaries \ref{corollary1} and \ref{corollary2} can be seen as extensions of results in \cite{Maeta and Ou} and \cite{Thesis}, because they show that if $f$ is constant on a non-empty open subset of $M$ then $f$ is constant on $M$.
\item Theorem~ \ref{theorem1} is meaningful even in the compact case.
\item
Consider $\varphi :M^{m}\hookrightarrow N^{m+1}(c), \  (c \leq 0)$ a proper-biharmonic hypersurface. Assume that $\grad f $ vanishes on an open subset. Then, it follows that $f$ is constant on an open (connected) subset. But, as $c\leq 0$, the constant has to be zero (see \cite{Oniciuc 2002} for a more general statement),
so $ \varphi $ is harmonic on an open subset, therefore on the whole manifold $M$.
\end{enumerate}
\end{remark}
As a direct application of Corollary  \ref{corollary2} we can give the following result. 

\begin{proposition} \label{Prop}
Let $ \varphi :M^{m}\hookrightarrow \mathbb{S}^{m+1}$ be a proper-biharmonic hypersurface  with constant scalar curvature. Assume that there exists a connected component of $M_{A}$ where the number of distinct
principal curvatures is  at most six. Then M has constant mean curvature.
\end{proposition}
\begin{proof}
Let $U$ be a connected component of $M_{A}$. The number of distinct principal curvatures is constant and at most $6$.

As $\Scal ^{M}$ is constant,  by Theorem 1.1 of \cite{fu-hong} we obtain that  $f$ is constant on $U$.
On the other hand, by Corollary {\ref{corollary2}}, we deduce that $f$ is constant on $M$.

\end{proof}

\begin{corollary}
Let $ \varphi :M^{m}\hookrightarrow \mathbb{S}^{m+1}$ be a proper-biharmonic  hypersurface. Assume that there exists a non-negative function $h$ such that $\vert \grad \vert A \vert ^{2} \vert \leq h \vert \grad f \vert$, and $M$ is not CMC. Denote by $U$ an open connected component of $M_{A}$. Then, on $U$ we have:
\begin{enumerate}
\item $-\frac{m}{2}f$ is a principal curvature with multiplicity equal to 1;
\item $\frac{\grad f }{\vert \grad f \vert }$ is a vector field defined on an open dense subset of $U$ and its integral curves are geodesics;
\item the number of distinct principal curvatures is at least  $3$ and $\vert A \vert ^{2}>\frac{m^{2}(m+8)}{4(m-1)}f^{2}$ on an open dense subset of $U$(see \cite{BMO}).
\end{enumerate}
\end{corollary}
\begin{proof}
Since $M$ does not have constant mean curvature, by Corollary \ref{Cor} we deduce that the points of $U$ where $\grad f \neq 0$ form an open dense subset of $U$. Now, by continuity we obtain
$-\frac{m}{2}f=\lambda_{i_{0}}$, for some $i_{0}$, on $U$, and by Proposition \ref{PR} we obtain that the multiplicity of $\lambda_{i_{0}}$ is 1.

Furthermore, for simplicity, we consider $i_{0}=1$, and work on an open connected subset of $U$ where $\grad f \neq 0$ at any point.
We have $E_{1}=\frac{\grad f} {\vert \grad f \vert}$ and taking the inner product of Equation (\ref{E1}) with $E_{1}$ for $i=1$ and $j=a$ we obtain 
$$ \omega_{1}^{a}(E_{1})=0,$$
and thus $\nabla_{E_{1}}E_{1}=0$.

If the number of distinct principal curvatures is at most $2$ then $U$ is CMC (see \cite{BMO}).
As J.-H.~Chen's Inequality  is based on the Cauchy-Schwarz Inequality applied to the principal curvatures, we have a strict inequality.

\end{proof}
\begin{remark}
We note that the distribution orthogonal to that determined by
$\frac{\grad f}{\vert \grad f\vert}$ is completely integrable. The level hypersurfaces of the mean curvature $f$  have flat normal connection as  submanifolds in $S^{m+1}$ of codimension $2$ (see  \cite[Theorem 1.40]{Nistor}).
\end{remark}

Corollary \ref{Cor} allows the re-writing of some known results
replacing their global hypothesis with local variants.
\begin{corollary}\label{Corollary}
Let $ \varphi :M^{m}\hookrightarrow \mathbb{S}^{m+1}$ be a proper-biharmonic  hypersurface.  Assume that  $ \vert \grad  \vert A\vert ^{2}\vert\leq h \  \vert \grad  f\vert$ on \textit{M}, where $h$ is a non-negative function on $M$. If $M$ is not CMC, then J.-H.~ Chen's Inequality
\begin{eqnarray}\label{A inequality}
\vert A \vert ^{2} \geq \frac{m^{2}(m+8)}{4(m-1)}f^{2} 
\end{eqnarray} 
is valid everywhere on M. 
\end{corollary}
\begin{proof}
Inequality (\ref{A inequality}) holds on $W$, and we conclude by continuity.

\end{proof}

J.-H.~Chen's Inequality enables us to obtain a more geometric version of Theorem ~ \ref{theorem1}.
\begin{theorem}
Let $ \varphi :M^{m}\hookrightarrow \mathbb{S}^{m+1}$ be a proper-biharmonic  hypersurface. Assume that $f^{2}> \frac{4(m-1)}{m(m+8)}$. If $ \grad  f$ vanishes on a non-empty open connected subset of \textit{M}, then $M$ has constant mean curvature.
\end{theorem}
\begin{proof}
Let us denote
$$ A_{0} :=\{ p \in M :( \grad f )(p)=0 \}.$$
In the proof of Theorem ~\ref{theorem1} we have shown that $A_{0}$ is a closed subset of $M$, $\Int A_{0}\neq \emptyset$, and if $\partial(\Int A_{0})= \emptyset$
then $\grad f $ vanishes on $M$.

As in the proof of Theorem ~\ref{theorem1}, assume that $\partial ({\Int A_{0}}) \neq \emptyset$, to reach a contradiction.
Let $p_{0}\in \partial (\Int A_{0})$, it follows that there exists a sequence of points $\{ p_{n}^{1}\}_{n\in \mathbb{N}^{*}}$  converging to $p_{0}$, $p_{n}^{1}\neq p_{0}$ and $p_{n}^{1}\in \Int A_{0}$ for any $n \in \mathbb{N}^{*}$, and there exists a sequence of points $\{p_{n}^{2}\}_{n\in \mathbb{N}^{*}}$, that converges to $p_{0}$, $p_{n}^{2}\neq p_{0}$ and $(\grad f )(p_{n}^{2})\neq 0$ for any $n\in \mathbb{N}^{*}$.

From Lemma (\ref{lemma I}) we have 
\begin{eqnarray}\label{h1}
\vert A\vert ^{2}(p_{n}^{2})\geq \frac{m^{2}(m+8)}{4(m-1)}f^{2}(p_{n}^{2}),\ \forall n \in \mathbb{N}^{*}.
\end{eqnarray} 
Now, each connected component of $\Int A_{0}$ is open in $\Int{A_{0}}$ and so in M. Thus, on each connected component of $\Int{A_{0}}$ the function $f$ is constant. 
But the constant cannot be zero as $\varphi$ is not harmonic and so $\vert A\vert ^{2}=m$. 
In conclusion, we have $\vert A \vert ^{2}=m$ on $\Int A_{0}$ and
\begin{eqnarray}\label{h2}
\vert A \vert ^{2} (p_{n}^{1})=m, \  \forall n \in \mathbb{N}^{*}.
\end{eqnarray}
Passing to the limit in (\ref{h1}) and  (\ref{h2}) we obtain
$$ m=\vert A \vert ^{2}(p_{0}) \geq \frac{m^{2}(m+8)}{4(m-1)}f^{2}(p_{0}),$$
thus $$ f^{2} \leq \frac{4(m-1)}{m(m+8)}$$
which is impossible.
\end{proof}
\begin{remark}
Compare the above result with \cite[Proposition 1.38 and Corollary 1.40]{Thesis}. 
\end{remark}
\section{Rigidity results for biharmonic hypersurfaces}\label{44}

The unique continuation properties of Section 2 can be exploited
to obtain new rigidity results. Theorem ~ \ref{M} relies essentially on the
combination of the Bochner formula applied to the vector field $\grad f$
and the J.-H.~ Chen's Inequality, made possible thanks to Corollary ~
\ref{Corollary}, while Theorem ~\ref{main'''} is a more technical alternative which puts
together a bound on the Ricci curvature and an averaged version of the
condition of \cite[Proposition 1.38]{Thesis}.

\begin{theorem}\label{M}
Let $ \varphi : M^{m} \hookrightarrow \mathbb{S}^{m+1}$ be a compact proper-biharmonic hypersurface. Assume that  $ \vert \grad \vert A \vert^{2} \vert \leq h  \ \vert \grad f \vert$ on $M$, where $h$ a non-negative function  on $M$, $\Scal^{M}\geq 0$ and
\begin{eqnarray}\label{BI}
\int_{M} \Big[m(m+8)f^{2}-4(m-1)\Big] \Scal^{M}f^{2} \ dv_{g}\geq 0.
\end{eqnarray}
Then M has constant mean curvature.
\end{theorem}
\begin{proof}
Assume that $M$ does not have constant mean curvature, we will argue by contradiction.

Starting with the Bochner Formula (see for example \cite{Petersen}), we have
$$-\frac{1}{2}\Delta \vert \grad f \vert^{2}= \vert \nabla d f \vert ^{2}-\langle \grad \Delta f,\grad f \rangle + \Ricci^{M}(\grad f,\grad f).$$
Now, by the Gauss Equation, we get
\begin{eqnarray}
\label{G}\\
\Ricci^{\mathbb{S}^{m+1}}(\grad f, \grad f)&=& \Ricci^{M}(\grad f, \grad f)+\vert A(\grad f)\vert ^{2} \nonumber
\\& \ & -mf \langle A(\grad f), \grad f\rangle +R^{\mathbb{S}^{m+1}}(\grad f,\eta, \grad f , \eta ),\nonumber
\end{eqnarray}
where $\eta$ is the unit normal vector field.

On the other hand, we have 
\begin{align*}
\Ricci^{\mathbb{S}^{m+1}}(\grad f , \grad f) &= 
 m\vert \grad f \vert ^{2}.
\end{align*}
Now, since  $M$ is a biharmonic submanifold of 
$\mathbb{S}^{m+1}$,

\begin{eqnarray*}
A(\grad f)&=&-\frac{m}{2}f\grad f,
\\
\vert A(\grad f )\vert ^{2} 
&=& \frac{m^{2}}{4}f^{2}\vert \grad f\vert ^{2}
,
\\
-mf\langle A(\grad f ), \grad f \rangle  &=& \frac{m^{2}}{2}f^{2}\vert  \grad f \vert^{2}.
\end{eqnarray*} 
Thus, using Equation (\ref{G}) we deduce that 
\begin{align*}
\Ricci^{M}(\grad f , \grad f )&= \Big(m-1-\frac{3m^{2}}{4}f^{2}\Big)\vert \grad f \vert ^{2}.
\end{align*}
Denote the scalar curvature $\Scal ^{M}$ by $s$. We have
\begin{eqnarray*}
\vert A \vert ^{2}&=&m(m-1)+m^{2}f^{2}-s.
\end{eqnarray*}
As $\varphi$ is biharmonic, we have
\begin{eqnarray*}
\Delta f &=&  (m- \vert A \vert ^{2})f,
\end{eqnarray*}
thus
\begin{eqnarray*}
\grad \Delta f &=& \grad [(m-\vert A \vert ^{2})f]
\\&=& m\grad f - f\grad \vert A \vert ^{2 }  -\vert A \vert ^{2} \grad f 
\\&=& m \grad f -f \grad (m^{2}f^{2}-s) -\vert A \vert ^{2}\grad f
\\&=& m \grad f -2 m^{2}f^{2}\grad f +f \grad s - \vert A \vert ^{2}\grad f 
\\ &=& (m-2m^{2}f^{2}-\vert A \vert ^{2})\grad f +f \grad s .
\end{eqnarray*} 
On the other hand, for a local orthonormal frame field $\{ e_{i} \}_{i=1}^{m}$,
\begin{eqnarray*}
\vert \nabla df \vert ^{2} &=& \sum_{i,j=1}^{m}\Big(\nabla df (e_{i},e_{j})\Big)^{2}
\\ &\geq& \sum _{i=1}^{m}\Big( \nabla df (e_{i},e_{i})\Big)^{2}
\\ &\geq& \frac{1}{m}\Big(\sum_{i=1}^{m} \nabla df (e_{i}, e_{i})\Big)^{2}
\\&\geq &\frac{1}{m}(\Delta f)^{2}.
\end{eqnarray*}
Now
\begin{eqnarray*}
-\frac{1}{2} \Delta \vert \grad f \vert ^{2} &\geq & \frac{1}{m}(\Delta f )^{2}-\langle(m-2m^{2}f^{2}-\vert A \vert ^{2})\grad f +f \grad s, \grad f \rangle 
\\ 
& \ & + \ \Big(m-1-\frac{3m^{2}}{4}f^{2}\Big)\vert \grad f \vert^{2} 
\\&\geq & \frac{1}{m}(\Delta f )^{2}+\Big(\vert A \vert ^{2}+\frac{5m^{2}}{4}f^{2}-1\Big)\vert \grad f\vert ^{2}-f\langle \grad s , \grad f \rangle . 
\end{eqnarray*}
We have
\begin{eqnarray*}
\frac{1}{m}\int _{M}(\Delta f )^{2} \ dv_{g}&=& \frac{1}{m} \int _{M} (\Delta f )(\Delta f) \ dv_{g}
\\&=& -\frac{1}{m} \int _{M}
(\Delta f) \Div (\grad f) \ dv_{g},\end{eqnarray*}
and using the Divergence Theorem we obtain
\begin{eqnarray*}
\frac{1}{m}\int _{M}(\Delta f)^{2} \, dv_{g} &=& \frac{1}{m} \int _{M}\langle \grad \Delta f, \grad f  \rangle \, dv_{g}
\\&=& - \ \frac{1}{m }\int_{M}(\vert A \vert ^{2}+2 m^{2}f^{2}-m)\vert \grad f \vert ^{2}\ dv_{g}
\\& \ & + \ \frac{1}{m}\int _{M}f\langle \grad s , \grad f\rangle \ dv_{g}.
\end{eqnarray*}
Integrating the Bochner Formula over $M$ and using the Divergence Theorem, we get 
\begin{align}
\label{INEQ}
0 &\geq -\frac{1}{m}\int _{M} (\vert A \vert ^{2}+2m^{2}f^{2}-m) \vert \grad f \vert ^{2} \ dv_{g}+\frac{1}{m} \int _{M}f\langle \grad s , \grad f\rangle \  dv_{g} {\nonumber}
\\ &\quad + \int _{M}
(\vert A \vert ^{2}+\tfrac{5m^{2}}{4}f^{2}-1)\vert \grad f \vert ^{2} \ dv_{g}-\int _{M}f\langle \grad s , \grad f\rangle \ dv_{g} {\nonumber}
\\ &\geq \int _{M}\Big[\Big(1-\tfrac{1}{m}\Big)\vert A \vert ^{2}+\Big(-\tfrac{2}{m}+\tfrac{5}{4} \Big) m^{2}f^{2}\Big]\vert \grad f \vert ^{2} \ dv_{g}
\\
& \quad +\Big( \frac{1-m}{m} \Big) \int _{M}f\langle \grad s , \grad f\rangle \ dv_{g}.\notag
\end{align}
To obtain a lower bound of the first term, we need Corollary \ref{Corollary} 
and we apply it to Equation (\ref{INEQ}) to obtain
\begin{eqnarray*}
0&\geq& \int _{M}\Big[\Big(\tfrac{m-1}{m}\Big)\Big(\tfrac{m^{2}(m+8)}{4(m-1)}\Big)f^{2}+\Big(\tfrac{5m-8}{4m}\Big)m^{2}f^{2}\Big)\Big]\vert \grad f \vert ^{2}\ dv_{g}
\\
&\ & + \ \Big( \frac{1-m}{2m} \Big)\int _{M} \langle \grad s, \grad f^{2}\rangle \  dv_{g}
\\
&\geq & \int _{M}\Big[  \tfrac{m(m+8)}{4}f^{2}+\tfrac{m(5m-8)}{4}f^{2}\Big]\vert \grad f\vert^{2} \ dv_{g}+\Big(  \frac{1-m}{2m}\Big)\int_{M}s \Delta f^{2}\ dv_{g}
\\
&\geq& \frac{3m^2}{2}\int_{M}f^{2}\vert \grad f \vert ^{2}\ dv_{g}+\Big( \frac{1-m}{2m}\Big)\int _{M}s\Delta f^{2}\ dv_{g}.
\end{eqnarray*}
Now,  we have $$ \Delta f^{2}=2 \Big( (m-\vert A \vert ^{2})f^{2}-\vert \grad f\vert ^{2} \Big),$$
thus 
\begin{align*}
&\frac{3m^{2}}{2}\int_{M} f^{2} \vert \grad f\vert ^{2}\ dv_{g}+\Big( \frac{1-m}{2m}\Big)\int_{M}2s\Big[(m-\vert A \vert ^{2})f^{2}-\vert \grad f \vert ^{2} \Big] \  dv_{g} \\&=
\frac{3m^{2}}{2}\int_{M}f^{2}\vert \grad f\vert ^{2}\ dv_{g}
+\Big(\frac{1-m}{m} \Big)\int_{M}s(m-\vert A \vert ^{2})f^{2}\ dv_{g}
\\ & \quad +\Big( \frac{m-1}{m} \Big)\int_{M}s \vert \grad f \vert ^{2}\ dv_{g}.
\end{align*}
Using Corollary \ref{Corollary} and the fact that $s\geq 0$, we obtain
\begin{align}\label{*}
0&\geq \frac{3m^{2}}{2}\int _{M}f^{2}\vert \grad f\vert ^{2} \ dv_{g}+(1-m)\int_{M}sf^{2}\ dv_{g} \nonumber
\\& \quad +
\frac{m(m+8)}{4}\int_{M}sf^{4}\ dv_{g}+\Big(\frac{m-1}{m} \Big) \int _{M}s\vert \grad f\vert ^{2}\ dv_{g}.
\end{align}
Now as $s\geq 0$, we obtain
\begin{eqnarray} \label{HI}
\int_{M}s\vert \grad f\vert ^{2}\ dv_{g} \geq 0
\end{eqnarray}
Then from Inequality (\ref{*}) we have
\begin{eqnarray}\label{N}
\\
0&\geq & \frac{3m^{2}}{2}\int _{M}f^{2}\vert \grad f\vert ^{2}\ dv_{g}+(1-m)\int _{M}sf^{2}\ dv_{g}+\frac{m(m+8)}{4}\int_{M}sf^{4} \ dv_{g}. \nonumber
\end{eqnarray}
Multiplying Inequality (\ref{N}) by $4$,
we obtain 
\begin{eqnarray*}
0&\geq & 6m^{2}\int_{M}f^{2} \vert \grad f\vert ^{2}\ dv_{g}+\int_{M}\Big[4(1-m) + m(m+8)f^{2}\Big]s f^{2}\ dv_{g}.
\end{eqnarray*}
Now as $\int _{M}[4(1-m) + m(m+8)f^{2}]s f^{2}\ dv_{g} \geq 0$ we obtain
\begin{eqnarray*}
0&\geq & 6m^{2}\int_{M}f^{2} \vert \grad f\vert ^{2}\ dv_{g}+\int_{M}\Big[4(1-m) + m(m+8)f^{2}\Big]s f^{2}\ dv_{g} \geq 0,
\end{eqnarray*}
by the sandwich rule we  conclude that 
\begin{align*}
6m^{2}\int_{M}f^{2} \vert \grad f\vert ^{2}\ dv_{g}+\int_{M}\Big[4(1-m) + m(m+8)f^{2}\Big]s f^{2}\ dv_{g}=0.
\end{align*}
Hence, at every point of $M$, $f^{2} \vert \grad f \vert^{2}=0$ which implies that, at each point of $M$   $f=0$ or $\grad f=0$.

Let $p\in M$ such that $ (\grad f)(p)\neq 0$, then $f=0 $ around $ p$, thus $\grad f=0 $ at $ p$, which is a contradiction.
Therefore, $\grad f=0$ at each point p, which implies that $f$ is constant everywhere on  $M$, and contradicts our assumption.

\end{proof}
 A weaker version of Theorem ~\ref{M} can be formulated, replacing the
condition on the scalar curvature by a combination of two lower bounds on the Ricci and scalar curvature, and an inequality involving the average  of 
the mean curvature.
\begin{theorem}\label{main'''}
Let $ \varphi : M^{m} \hookrightarrow \mathbb{S}^{m+1}$ be a compact proper-biharmonic hypersurface. Assume that there exist a non-negative function  $h$ on $M$ such that $ \vert \grad \vert A \vert^{2} \vert \leq h  \ \vert \grad f \vert$ on $M$, and a real number $a>0$ such that:
\begin{enumerate}
\item $\Ricci^{M}(X,X) \geq a > 0,$ for all  $ X \in T_{p}M,\  \vert X \vert =1 \ and$  for all $p \in M;$\label{Cdt1'''}
\item $\int _{M}[m^{2}(m+8)af^{2}-4(m-1)s
 ]f^{2} \ dv_{g} \geq 0.$ \label{Cdt2'''}
\end{enumerate}
Then M has constant mean curvature.
\end{theorem}
\begin{proof}
Proceeding in the exact same way as in the proof of the previous theorem,
we reach
\begin{align}\label{*'''}
0&\geq \frac{3m^{2}}{2}\int _{M}f^{2}\vert \grad f\vert ^{2} \ dv_{g}+(1-m)\int_{M}sf^{2}\ dv_{g} \nonumber
\\& \quad +
\frac{m(m+8)}{4}\int_{M}sf^{4}\ dv_{g}+\Big(\frac{m-1}{m} \Big) \int _{M}s\vert \grad f\vert ^{2}\ dv_{g}.
\end{align}
To control the terms in $f^{2}$, for the Hilbert space $L^{2}(M)$, we consider
an orthonormal basis $\{ f_{i}\}_{i=0}^{\infty}$ of $C^\infty(M)$-eigenfunctions of the Laplacian, i.e. $\Delta f_{i}= \lambda _{i}f_{i}$,  where $\lambda_{0}=0 < \lambda _{1} \leq \lambda _{2}\leq \cdots$, and
$\int _{M}f_{i}f_{j} \  dv_{g}=\delta_{ij}$.
\\

Let $f\in C^\infty(M)$, then $f=\sum_{i=0}^{\infty}\mu_{i}f_{i}$, where $f_{0}=\frac{1}{\sqrt{\Vol(M)}}$ and  $\mu_{0}=\frac{1}{\sqrt{\Vol(M)}} \int_{M}f \ dv_{g}. $
\\
Then
\begin{align*}
\int_{M}f^{2} \ dv_{g}=\sum_{i=0}^{\infty}\mu_{i}^{2}.
\end{align*}
Also,
\begin{align*}
\Delta f =\sum_{i=0}^{\infty}\lambda _{i}\mu _{i}f_{i}=\sum_{i=1}^{\infty}\mu _{i}\lambda _{i}f_{i},
\end{align*}
so we have 
\begin{eqnarray*}
\int _{M}f\Delta f \ dv_{g}
&=&\sum_{i=1}^{\infty} \lambda _{i}\mu_{i}^{2} 
\\
&\geq & \lambda _{1}\sum_{i=1}^{\infty}\mu _{i}^{2} =
\lambda _{1} \Big( \int_{M}f^{2}\ dv_{g}- \mu _{0}^{2} \Big).
\end{eqnarray*}
But 
\begin{align*}
\int_{M}f\Delta f \ dv_{g}= \int _{M} \vert \grad f\vert ^{2} \ dv_{g},
\end{align*} 
so
\begin{eqnarray*}
\int_{M}\vert \grad f\vert ^{2} \ dv_{g} &\geq & \lambda _{1} \Big[ \int_{M}f^{2}\ dv_{g}-\frac{1}{\Vol(M)}\Big( \int _{M} f \ dv_{g} \Big)^{2} \Big].
\end{eqnarray*}
Now, by Obata \cite{Obata},  $\Ricci^{M}(X,X)\geq a|X|^{2}>0$ implies that $\lambda _{1}\geq \frac{ma}{m-1} .$
Since $s \geq ma$, we have 
\begin{eqnarray}
\label{HIB}
\int_{M}s\vert \grad f\vert ^{2}\ dv_{g} &\geq & ma\int_{M}\vert \grad f\vert ^{2} \ dv_{g}
\\
&\geq & \frac{m^{2}a^{2}}{m-1}\Big[ \int_{M}f^{2} \ dv_{g}-\frac{1}{\Vol(M)}\Big( \int_{M}f  \ dv_{g} \Big)^{2} \Big].\nonumber
\end{eqnarray}
Then from Inequality (\ref{*'''})
\begin{eqnarray}
0
&\geq &\frac{3m^{2}}{2}\int_{M}f^{2}\vert \grad f\vert ^{2}\ dv_{g}+(1-m)\int _{M}sf^{2} \ dv_{g}+\frac{m(m+8)}{4}\int_{M}sf^{4}\ dv_{g} \nonumber
\\
&\ & + \ ma^{2}\int_{M}f^{2} \ dv_{g}-\frac{ma^{2}}{\Vol(M)}\Big( \int _{M}f \ dv_{g}\Big)^{2},\label{N'''}
\end{eqnarray}
and multiplying Inequality (\ref{N'''}) by $4$, and using $s\geq ma$ we obtain 
\begin{eqnarray*}
0 &\geq & 6m^{2}\int _{M} f^{2}\vert \grad f\vert ^{2}\ dv_{g}+4(1-m)\int_{M}sf^{2} \ dv_{g}+m^{2}(m+8)a \int_{M}f^{4} \ dv_{g}
\\
&\ & +\ 4ma^{2}\int_{M}f^{2} \ dv_{g}-\frac{4ma^{2}}{\Vol(M)}\Big( \int_{M}f \ dv_{g} \Big)^{2}
\\
&\geq & 6m^{2}\int_{M}f^{2} \vert \grad f\vert ^{2}\ dv_{g} +4\int _{M}\Big[(1-m)s+ma^{2}\Big]f^{2} \ dv_{g}
\\
&\ & + \  m^{2}(m+8)a \int_{M}f^{4} \ dv_{g}-\frac{4ma^{2}}{\Vol(M)}\Big( \int_{M}f \ dv_{g} \Big)^{2}. 
\end{eqnarray*}
Now, using the Cauchy-Schwarz Inequality, we obtain
\begin{align*}
0\geq
\  & 6m^{2}\int_{M}f^{2} \vert \grad f\vert ^{2}\ dv_{g} 
\\
&+ \int _{M}\Big[4(1-m)s+4ma^{2}
+ \    m^{2}(m+8)af^{2}-4ma^{2}\Big]f^{2} \ dv_{g}.
\end{align*}
By Condition \ref{Cdt2'''} we  conclude that 
\begin{align*}
6m^{2}\int_{M}f^{2} \vert \grad f\vert ^{2}\ dv_{g} +\int _{M}\Big[4(1-m)s
+ \    m^{2}(m+8)af^{2}\Big]f^{2} \ dv_{g} =0,
\end{align*}
to prove Theorem~\ref{main'''}.

\end{proof}
\begin{remark}
The conditions in Theorems~\ref{M} 
and~\ref{main'''} are satisfied by
the 45th-parallel $\mathbb{S}^{m}(\frac{1}{\sqrt{2}})\hookrightarrow \mathbb{S}^{m+1}$ ($m\geq 2$) since
 $\Scal=2m(m-1)$, $f^{2}=1$ and we can take $a=2(m-1)$.

The generalized Clifford torus $ \mathbb{S}^{m_{1}}(\frac{1}{\sqrt{2}})\times\mathbb{S}^{m_{2}}(\frac{1}{\sqrt{2}})\hookrightarrow \mathbb{S}^{m+1}$ ($m_{1}<m_{2}$) has $ \Scal=2m_{1}(m_{1}-1)+2m_{2}(m_{2}-1)$, and $f^{2}=\Big( \frac {m_{1}- m_{2}}{m_{1}+m_{2}}\Big)^{2}$, by straightforward calculations it satisfies Condition \ref{BI} in Theorem~\ref{M} if and only if $m_{1} \in \Big[1,\frac{m}{2}-\sqrt{\tfrac{m(m-1)}{m+8}}\Big]$, ($m=m_{1}+m_{2}$) which includes  $m_{2}\geq 2m_{1}+2.$
For $m$ large enough, there always exist  Clifford tori satisfying the conditions of Theorem~\ref{main'''}, for $a=2(m_{1}-1)$, $m_{1}\geq 2$.

\end{remark}

\end{document}